\documentclass{amsart}
\usepackage{graphicx}

\def\cal{\mathcal}

\input{tcilatex}

\newcommand{\EE}{{\mathbb E}}

\newcommand{\NN}{{\mathbb N}}
\newcommand{\PP}{{\mathbb P}}

\newcommand{\vX}{\overleftarrow{X}}
\newcommand{\vtau}{\overleftarrow{\tau}}
\newcommand{\htau}{\widehat{\tau}}
\newcommand{\wtau}{\widetilde{\tau}}

\newcommand{\hX}{\widehat X}
\newcommand{\hP}{\widehat P}

\newcommand{\wP}{\widetilde P}

\newcommand{\wI}{\widetilde I}

\newcommand{\uP}{\underline P}
\newcommand{\wX}{\widetilde X}
\newcommand{\wx}{{\widetilde x}}
\newcommand{\wy}{{\widetilde y}}
\newcommand{\wmu}{{\widetilde \mu}}
\newcommand{\umu}{{\underline \mu}}
\newcommand{\uX}{{\underline X}}
\newcommand{\upartial}{{\underline{\partial}}}

\newcommand{\wpartial}{{\widetilde{\partial}}}

\newcommand{\bpi}{{\overline{\pi}}}

\input tcilatex

\begin{document}

\title[duality and intertwining]{Truncation in Duality and Intertwining Kernels}
\author{Thierry Huillet$^{1}$, Servet Martinez$^{2}$}
\address{$^{1}$Laboratoire de Physique Th\'{e}orique et Mod\'{e}lisation \\
CNRS-UMR 8089 et Universit\'{e} de Cergy-Pontoise, 2 Avenue Adolphe Chauvin,
95302, Cergy-Pontoise, FRANCE\\
$^{2}$ Departamento de Ingenier\'{i}a Matem\'{a}tica \\
Centro Modelamiento Matem\'{a}tico\\
UMI 2807, UCHILE-CNRS \\
Casilla 170-3 Correo 3, Santiago, CHILE.\\
E-mail: Thierry.Huillet@u-cergy.fr and smartine@dim.uchile.cl}
\maketitle

\begin{abstract}
We study properties of truncations in the dual and intertwining process in
the monotone case. The main properties are stated for the time-reversed
process  and the time of absorption of the truncated intertwining process.
\end{abstract}

\textbf{Running title}: Duality and Intertwining.\newline



\section{Introduction}

In this work we study truncation of stochastic kernels and their relation to
duality and intertwining.

In next section we recall these concepts, a main one being that when one
starts from an stochastic and positive recurrent matrix and one construct a
dual, then the associated intertwining matrix is associated to the
time-reversed matrix of the original one.

Thus, many of the concepts that are proven for the original matrices need
to be shown for the time-reversed matrix. This gives rise to a problem 
when dealing
with monotonicity because it is not necessarily invariant by 
time-reversing. On the
other hand monotonicity property plays a central role in duality and
intertwining because a nonnegative dual kernel exists for the Siegmund dual
function, and in this case several properties can be stated between the
original kernel and the dual and intertwining associated kernels.

In this framework our first result is to extend the Pollak-Siegmund relation
stated for the monotone process to their time-reversed process. In terms of
truncation this result asserts that whatever is the truncation, when the
level increases, the quasi-stationary distribution converges to the  
stationary distribution.
This is done in Proposition \ref{propx2} in Section \ref{Section11}.

In Section \ref{Section12} we introduce the truncation given by the mean
expected value, which satisfies that the truncation up to $N$ preserves the
stationary distribution up to $N-1$. In Proposition \ref{propx4} it is
shown that it has a nonnegative dual, and in Proposition \ref{propx5} it is
proven that the time of attaining the absorbing state is increasing with the
level, either for the dual and the intertwining matrices. For the
Diaconis-Fill coupling this means that the time for attaining the stationary
distribution is stochastically increasing with the truncation of the
reversed process.

In Proposition \ref{propx7} we compare the quasi-stationary behavior for the
Diaconis-Fill coupling with respect to the intertwining kernel, when in both
processes one avoids the absorbing state.

We recall some of the main general results in duality and intertwining in
Section \ref{SectionDI} and for the monotone case and the Siegmund kernel
this is done in Section \ref{SectionMS}. The strong stationary times and
some of its properties are given in Section \ref{SectionSST}, and the
Diaconis-Fill coupling is given in Section \ref{SectionDFC}.

\section{Duality and Intertwining}

\subsection{Notation}
\label{s1} 

Let $I$ be a countable set and $\mathbf{1}$ be the unit function
defined on $I$. A nonnegative matrix $P=(P(x,y):x,y\in I)$ is called a
kernel on $I$. A substochastic kernel is such that $P\mathbf{1}\le \mathbf{1}
$, it is stochastic when $P\mathbf{1}=\mathbf{1}$, and strictly
substochastic if it is substochastic and there exists some $x\in I$ such
that $P\mathbf{1}(x)<1$.

A kernel $P$ is irreducible if for any pair $x,y\in I$ there exists $n>0$
such that $P^{n}(x,y)>0$. A point $x_{0}\in I$ is absorbing for $P$ when 
$P(x_{0},y)=\delta_{y,x_{0}}$ for $y\in I$ ($\delta_{x_0,x_{0}}=1$ and 
$\delta _{y,x_{0}}=0$ otherwise).

Let $P$ be substochastic. Then, there exists a Markov chain 
$X=(X_{n}:n<\mathcal{T}^{X})$ uniquely defined in distribution, taking 
values in $I$,
with lifetime $\mathcal{T}^{X}$ and with transition kernel $P$. Take 
$\partial\not{\in }I$ and define $X_n=\partial$ for all 
$n\ge \mathcal{T}^{X}$. $P$ acts on the set of bounded or nonnegative 
real functions by 
$Pf(x)={\mathbb E}(f(X_{1})\, \mathbf{1}(\mathcal{T}^{X}>1))$ and 
$P^{n}f(x)={\mathbb E}(f(X_{n})\, \mathbf{1}(\mathcal{T}^{X}>n))$ 
for $n\ge 1$, $x\in I$.

In the sequel, we introduce three kernels, $P$, $\widehat{P}$, $\widetilde{P}
$; defined  respectively on the countable sets $I$, $\widehat{I}$, 
$\widetilde{I}$. When they are substochastic the associated Markov chains are
respectively denoted by $X$ $\widehat{X}$, $\widetilde{X}$ with lifetimes 
$\mathcal{T},\widehat{\mathcal{T}},\widetilde{\mathcal{T}}$. For the chain 
$X$ we note by $\tau _{K}=\inf \{n\ge 0:X_{n}\in K\}$ the hitting time of 
$K\subseteq I$ and put $\tau _{a}=\tau _{\{a\}}$ for $a\in I$. Analogously,
we define for $\widehat{X}$ (respectively for $\widetilde{X}$) the hitting
times ${\widehat{\tau }}_{\widehat{K}}$ and ${\widehat{\tau }}_{\widehat{a}}$
(respectively ${\widetilde{\tau }}_{{\widetilde{K}}}$ and 
${\widetilde{\tau }}_{\widetilde{a}}$).

\subsection{Definitions}

We recall the duality and the intertwining relations. As usual $M^{\prime}$
denotes the transpose of the matrix $M$, so $M^{\prime}(x,y)=M(y,x)$ for 
$x,y\in I$.

\begin{definition}
\label{def1} 
Let $P$ and $\widehat{P}$ be two kernels defined on the
countable sets $I$ and $\widehat{I}$, and let $H=(H(x,y):x\in I,y\in 
\widehat{I})$ be a matrix. Then $\widehat{P}$ is said to be a $H-$dual of 
$P$, and $H$ is called a duality function between $(P,\widehat{P})$, if it is
satisfied 
\begin{equation}
\label{eqxy1}
H\widehat{P}^{\prime }=PH.  
\end{equation}
$\Box $
\end{definition}

Duality is symmetric because if $\widehat{P}$ is a $H-$dual of $P$, then $P$
is a $H^{\prime }-$dual of $\widehat{P}$. We only consider nonnegative
duality functions $H$. Note that if $H$ is a duality function between 
$(P,\widehat{P})$, then $cH$ also is for all $c>0$. We assume that no 
row and no column of $H$ vanishes completely.

If $\widehat{P}$ is a $H-$dual of $P$, then $\widehat{X}$ is a $H-$dual of 
$X$, in the following sense 
$$
\forall \,x\in I,\,y\in \widehat{I},\,\forall \,n\ge 0\,: 
{\mathbb E}_{x}(H(X_{n},y))={\mathbb E}_{y}(H(x,\widehat{X}_{n})), 
$$
where $H$ is extended to $(I\cup \{\partial \})\times (\widehat{I}\cup
\{\partial \})$ by putting $H(x,\partial ) =H(\partial,y)
=H(\partial,\partial )=0$ for all $x\in I$, $y\in \widehat{I}$.

This notion of duality (\ref{eqxy1}) coincides with the one between Markov
processes found in references \cite{Lig}, \cite{MM} and \cite{DF}, among
others. Let us now introduce intertwining as in \cite{cpy}.

\begin{definition}
\label{def2} 
Let $P$ and $\widetilde{P}$ be two kernels defined on the
countable sets $I$ and $\widetilde{I}$ and let $\Lambda =(\Lambda (y,x):y\in 
\widetilde{I},x\in I)$ be a stochastic matrix. We say that $\widetilde{P}$
is a $\Lambda -$intertwining of $P$, and $\Lambda $ is called a link between 
$(P,\widetilde{P})$, if it is satisfied 
$$
\widetilde{P}\Lambda =\Lambda P.
$$
$\Box $
\end{definition}

Intertwining is not necessarily symmetric because $\Lambda^{\prime}$ is not
necessarily stochastic. If $\Lambda$ is doubly stochastic then $\widetilde{P}
$ a $\Lambda-$ intertwining of $P$ implies $P$ is a 
$\Lambda^{\prime}-$intertwining of $\widetilde{P}$. 
If $\widetilde{P}$ is a $\Lambda-$intertwining of $P$, 
then $\widetilde{X}$ is said to be a $\Lambda-$intertwining of $X$.

Throughout this paper we consider $I={\mathbb N}=\{1,2,...\}$, the set of
nonnegative integers in the infinite case or $I=I_{N}=\{1,\cdots ,N\}$ 
with $N\ge 2$ in the finite case.

We note by $\mathbf{e}_{a}$ a column vector with $0$ entries except for its 
$a-$th entry which is $1$.

For a vector $\rho \in \mathbf{R}^{I}$ we denote by $D_{\rho }$ the diagonal
matrix with diagonal entries $(D_{\rho })(x,x)=\rho (x)$, $x\in I$.

\textbf{Assumption.} From now on, $P=(P(x,y): x,y\in I)$ is assumed to be an
irreducible positive recurrent stochastic kernel and its stationary
distribution is noted by $\pi=(\pi(i): i\in I)>0$.

Let $\overleftarrow{P}$ be the time-reversed transition kernel of $P$, so 
$\overleftarrow{P}^{\prime}=D_{\pi }PD_{\pi }^{-1}$ Since 
$\overleftarrow{P}$
is also irreducible positive recurrent with stationary distribution 
$\pi^{\prime}$, we can exchange the roles of $P$ and $\overleftarrow{P}$.

\begin{remark}
\label{remk1} 
Let $\widetilde{P}$ be a $\Lambda -$intertwining of $P$. Since 
$\Lambda $ is stochastic, when $\widetilde{\beta }^{\prime }$ is a
stationary probability measure of $\widetilde{P}$ then 
$\widetilde{\beta }^{\prime }\Lambda $ is a stationary probability measure 
of $P$. Therefore, if $\widetilde{a}$ is an absorbing state for $\widetilde{P}$, 
then we necessarily have, 
\begin{equation}
\label{eqxy2}
\pi^{\prime }=\mathbf{e}_{\widetilde{a}}^{\prime }\Lambda. 
\end{equation}
$\Box $
\end{remark}

\section{Relations}
\label{SectionDI}

Below we supply Theorem \ref{theo1} shown in \cite{HM1}
which summarizes several relations on duality and intertwining.

\begin{theorem}
\label{theo1} 
Let $P$ be an irreducible positive recurrent stochastic kernel
with stationary distribution $\pi ^{\prime }$. Assume $\widehat{P}$ is a
kernel which is $H-$dual of $P$, $H\widehat{P}^{\prime }=PH$. Then:

$(i)$ $\widehat{P}H^{\prime }D_{\pi }=H^{\prime }D_{\pi }\overleftarrow{P}$;

$(ii)$ $\varphi :=H^{\prime }\pi $ is strictly positive and satisfies 
$\widehat{P}\varphi =\varphi $;

When $\widehat{P}$ is stochastic and irreducible then $\varphi =c\mathbf{1}$
for some $c>0$ and $\widetilde{P}=\widehat{P}$.

$(iii)$ $\widetilde{P}=D_{\varphi }^{-1}\widehat{P}D_{\varphi }$ is a
stochastic kernel (defined on $\widetilde{I}=\widehat{I})$ and the matrix 
$\Lambda :=D_{\varphi }^{-1}H^{\prime }D_{\pi }$ is stochastic. Moreover 
$\widetilde{P}$ is a $\Lambda -$intertwining of $\overleftarrow{P}$. Hence,
it holds 
$$
\widetilde{P}\Lambda =\Lambda \overleftarrow{P}\, 
\text{ satisfying } 
\widetilde{P}\mathbf{1}=\mathbf{1}=\Lambda \mathbf{1}.
$$

Now assume $I$ and $\widehat{I}$ are finite sets.

$(iv)$ If $\widehat{P}$ is strictly substochastic then it is not irreducible.

$(v)$ If $\widehat{P}$ is strictly substochastic and has a unique stochastic
class $\widehat{I}_{\ell }$, then 
$$
\frac{\varphi (x)}{\varphi (y)}=
{\mathbb P}_{x}(\widehat{\tau }_{\widehat{I}_{\ell}}
<\widehat{\mathcal{T}})\;\text{for any }y\in \widehat{I}_{\ell },
$$
and the intertwined Markov chain $\widetilde{X}$ is given by the Doob
transform 
\begin{equation}
\label{eqxy3}
{\mathbb P}_{x}(\widetilde{X}_{1}=y_{1},\cdots ,\widetilde{X}_{k}
=y_{k})={\mathbb P}_{x}(\widehat{X}_{1}=y_{1},\cdots ,
\widehat{X}_{k}=y_{k}\,|\, \widehat{\tau }_{\widehat{I}_{\ell }}
<\widehat{\mathcal{T}}).  
\end{equation}

$(vi)$ If $\widehat{a}$ is an absorbing state for $\widehat{P}$ then 
$\widehat{a}$ is an absorbing state for $\widetilde{P}$ and the relation 
(\ref{eqxy2}) $\pi^{\prime }=\mathbf{e}_{\widehat{a}}^{\prime }\Lambda $, 
is satisfied.

$\Box $
\end{theorem}

One has 
$$
\Lambda (x,y)=\varphi(x)^{-1} H(y,x) \pi(y), 
$$
in particular $\Lambda (x,y)=0$ if and only if $H(y,x)=0$.

Notice that when $\varphi >0$, the equality $\widetilde{P}=
D_{\varphi }^{-1}\widehat{P}D_{\varphi }$ implies that the sets of 
absorbing points for $\widehat{P}$ and $\widetilde{P}$ coincide.

\begin{remark}
\label{remk2} 
When the starting equality between stochastic kernels is the
intertwining relation $\widetilde{P}\Lambda =\Lambda \overleftarrow{P}$,
then we have the duality relation $H\widehat{P}^{\prime }=PH$ with 
$H=D_{\pi}^{-1}\Lambda ^{\prime }$ and $\widehat{P}=\widetilde{P}$. In 
this case $\varphi=\mathbf{1}$. $\Box $
\end{remark}

The next result of having a constant column appears as a condition in the
study of sharp duals, see Remark \ref{rmk5} in Section \ref{SectionSST}.

\begin{proposition}
\label{propx1} 
Assume $H$ is nonsingular and has a strictly positive
constant column, that is 
$$
\exists \widehat{a}\in \widehat{I}:H\mathbf{e}_{\widehat{a}}=
c\mathbf{1}\,\text{for some }\,c>0.
$$
Then:

$(i)$ $\widehat{a}$ is an absorbing state for $\widehat{P}$ 
(so $\{\widehat{a}\}$ is a stochastic class).

$(ii)$ Under the hypotheses of Theorem \ref{theo1}, $\pi^{\prime }=
\mathbf{e}_{\widehat{a}}^{\prime }\Lambda $ holds and if $\widehat{P}$ is strictly
substochastic and $\{\widehat{a}\}$ is the unique stochastic class then 
${\mathbb P}_{y}(\widehat{\tau }_{\widehat{a}}<\widehat{\mathcal{T}})=
\varphi(y)/\varphi (\widehat{a})$ and the relation (\ref{eqxy3}) is 
satisfied. $\,\Box $
\end{proposition}

\begin{remark}
\label{remk3} 
Duality functions with constant columns appear in the
following situations. If $x_{0}\in I$ is an absorbing point of the kernel $P$
and $\widehat{P}$ is a substochastic kernel that is a $H-$dual of $P$, then 
$h(y):=H(x_{0},y)$, $y\in \widehat{I}$, is a nonnegative 
$\widehat{P}-$harmonic function. So, when $H$ is bounded and 
$\widehat{P}$ is a stochastic
recurrent kernel, the $x_{0}-$row $H(x_{0},\cdot )$ is constant. $\Box $
\end{remark}

\begin{remark}
\label{remk33}
The time-reversed transition kernel $\overleftarrow{P}$ can always 
be put as a Doob transform
$\overleftarrow{P}=D_{\varphi} P^{\prime} D_{\varphi }^{-1}$ with
$\varphi(x)=\PP_x(\tau>1)$ for some stopping time $\tau$. In this case 
$\varphi=\pi$ so we must only define $\tau$. For the Markov chain 
$X=(X_n)$ with transition matrix $P$ it exists a 
set of random function $(U_{n,x}: n\ge 1, i\in I)$ taking values 
in $I$ and independent for different $n$, such that 
$X_{n+1}=U_{n+1,X_n}$ for all $n\ge 0$. 
Now take a collection of independent identically 
distributed random variables $({\cal J}_n: n \ge 1)$
distributed as $\pi$,  
and define the Markov chain $Y=(Y_n: n\ge 0)$ by
$$
Y_{n+1}={\cal J}_{n+1} \cdot {\bf 1}(U_{n+1,{\cal J}_{n+1}}=Y_n)+
\partial \cdot {\bf 1}(U_{n+1,{\cal J}_{n+1}}\neq Y_n), \, n\ge 0,
$$
where $\partial\not\in I$ is an absorbing state for $Y$.
For $x,y\in I$ we have $\PP(Y_{n+1}=y, Y_n=x)=\pi(y)\PP(U_{n+1,y}=x)=\pi(y) P(y,x)$.
Let $\tau=\inf(n\ge 1: Y_n=\partial)$, then we have $\PP_x(\tau>1)=\pi(x)$ and
so $\PP_x(Y_{1}=y \, | \, \tau>1)=\pi(y) P(y,x) \pi(x)^{-1}$. Then 
$(Y_n: n\ge 0)$ is a Markov chain such that for all $n\ge 1$ one has 
$\PP_x(Y_{n}=y \, | \, \tau>n)= \overleftarrow{P}^{(n)}(x,y)$.
\end{remark}

\section{Monotonicity and the Siegmund kernel}
\label{SectionMS} 

The Siegmund kernel, see \cite{Sieg}, is defined by 
$$
H^S(x,y)=\mathbf{1}(x\le y), x,y\in I.
$$
This kernel is nonsingular and $(H^S)^{-1}=Id-R$ with 
$R(x,y)=\mathbf{1}(x+1=y)$ (which is a strictly substochastic kernel 
because the $N-$th row vanishes), so 
$(H^S)^{-1}(x,y)=\mathbf{1}(x=y)-\mathbf{1}(x+1=y)$ and 
$H^S=(Id-R)^{-1}$ is the potential matrix associated to $R$. In \cite{Sieg}
the Siegmund duality was used to show the equivalence between absorbing and
reflecting barrier problems for stochastically monotone chains.

Let $\widehat{P}$ be a substochastic kernels such that 
$H^{S}\widehat{P}^{\prime}=PH^{S}$. Since 
$(H^{S}\widehat{P}^{\prime })(x,y)=\sum_{z\ge x}\widehat{P}(y,z)$ 
and $(PH^{S})(x,y)=\sum_{z\le y}P(x,z)$, they must satisfy 
\begin{equation}
\label{eqxy4}
\widehat{P}(y,x)=\sum\limits_{z\ge x}\widehat{P}(y,z)-
\sum\limits_{z>x}\widehat{P}(y,z)
=\sum\limits_{z\le y}(P(x,z)-P(x+1,z)),\;x,y\in I.
\end{equation}
In particular, the condition $\widehat{P}\ge 0$ requires the monotonicity 
of $P$, 
$$
\forall y\in I\,:\;\sum\limits_{z\le y}P(x,z)\text{ decreases in }\,x\in I.
$$
From $P(1,1)<1$ one gets that $\widehat{P}$ loses mass through $1$.
Moreover, if $r$ is the smallest integer such that $\sum_{z\le r}P(1,z)=1$,
then $\widehat{P}$ loses mass through $\{x<r\}$ and it does not lose mass
through $\{r,\cdots ,N\}$. By applying Theorem \ref{theo1} one gets 
$$
\varphi (x)=(H^{S})^{\prime }\pi (x)=\sum\limits_{y\in I}\mathbf{1}
(y\le x)\pi (y)=\sum\limits_{y\le x}\pi (y)=:\pi ^{c}(x),
$$
the cumulative distribution of $\pi $, which is not constant because 
$\pi>0$.

Consider the finite case with $I=\widehat{I}=\widetilde{I}=I_{N}$, $N\ge 2$.
If $P$ is substochastic then the equality $\widehat{P}(N,x)=
\sum\limits_{z\le N}(P(x,z)-P(x+1,z))$ 
implies $\widehat{P}(N,x)=\delta _{x,N}$, so $N$ is
an absorbing state for $\widehat{P}$. It can be checked that $N$ is the
unique absorbing state for $\widehat{P}$. We can summarize the above
analysis by the following result.

\begin{corollary}
\label{coro1} 
Let $H^{S}$ be the Siegmund kernel, $P$ be a monotone
irreducible positive recurrent stochastic kernel with stationary
distribution $\pi ^{\prime }$. Let $H^{S}\widehat{P}^{\prime }=PH^{S}$ with 
$\widehat{P}\ge 0$. Then:

$(i)$ $\varphi =\pi ^{c}$ and the stochastic intertwining kernel $\Lambda$
satisfies 
\begin{equation}
\label{eqxy5}
\Lambda (x,y)=\mathbf{1}(x\ge y)\frac{\pi (y)}{\pi ^{c}(x)}. 
\end{equation}
The intertwining matrix $\widetilde{P}$ of $\overleftarrow{P}$, that
verifies $\widetilde{P}\Lambda =\Lambda \overleftarrow{P}$, is given by 
\begin{equation}
\label{eqxy6}
\widetilde{P}(x,y)=\widehat{P}(x,y)\frac{\pi ^{c}(y)}{\pi^{c}(x)},
\;x,y\in I. 
\end{equation}

Now assume $I=I_{N}$, then,

$(ii)$ $\widehat{P}$ is strictly substochastic and loses mass through $1$,
and parts $(iv)$, $(v)$ and $(vi)$ of Theorem \ref{theo1} hold.

$(iii)$ $N$ is the unique absorbing state for $\widehat{P}$ (and for 
$\widetilde{P}$), Theorem \ref{theo1} parts $(v)$ and $(vi)$ are fulfilled
with $\widehat{I}_{\ell }=\{N\}$ and $\widehat{a}=N$. In particular 
$\pi^{\prime }=\mathbf{e}_{N}^{\prime }\Lambda $ holds.

$(iv)$ The following relation is satisfied 
\begin{equation}
\label{eqxy7}
\Lambda \mathbf{e}_{N}=\pi (N)\mathbf{e}_{N}\,. 
\end{equation}
$\Box $
\end{corollary}

\section{Truncation for Monotone kernels}
\label{Section1}

The purpose of this section is to see how the truncations
behaves with the duality relation.

\subsection{The Pollak-Siegmund limit for the reversed chain}
\label{Section11}

Since $P$ has stationary distribution $\pi ^{\prime }$ it holds 
$\lim\limits_{n\to \infty }{\mathbb P}_{x}(X_{n}=y)=\pi (y)$ for 
$y\in {\mathbb N}$. Pollak and Siegmund proved in \cite{PS} that 
if $P$ is also monotone then, 
\begin{equation}
\label{eqxy8}
\forall x,y\in {\mathbb N}:\quad \lim\limits_{n,N\to \infty }
{\mathbb P}_{x}(X_{n}=y\,|\,\tau _{(N)}>n)=\pi (y),  
\end{equation}
where $\tau_{(N)}$ is the hitting times of the domain $R_{N}=\{z\ge N\}$ by
the chain $X$. So, a truncation at a sufficiently high level, will have a
stationary distribution close to the one of the original process.

Now, in the framework of Theorem \ref{theo1} the intertwining relation 
$\widetilde{P}\Lambda =\Lambda \overleftarrow{P}$ is constructed from a
duality relation, and so $\overleftarrow{P}$ plays the role of $P$. This
leads us to show the Pollak-Siegmund relation for the reversed kernel 
$\overleftarrow{P}(x,y)=\pi (x)^{-1}P(y,x)\pi (y)$.

\medskip

First note that for any path $x_{0},...,x_{n}$ it holds 
\begin{equation}
\label{eqxy9}
\prod_{k=0}^{n-1}\overleftarrow{P}^{n}(x_{k},x_{k+1})=
\pi(x_{0})^{-1}
\left(\prod_{k=0}^{n-1}P^{n}(x_{k+1},x_{k})\right) \pi (x_{n}). 
\end{equation}
Let $\overleftarrow{\tau }_{(N)}$ be the hitting times of $R_{N}=\{z\ge N\}$
by the chain $\overleftarrow{X}$. From (\ref{eqxy9}) one gets 
\begin{eqnarray*}
&{}&{\mathbb P}_{x}(\overleftarrow{X}_{n}=z,
\overleftarrow{\tau }_{(N)}>n)=\pi(x)^{-1}\pi(z)
{\mathbb P}_{z}(\tau _{(N)}>n,X_{n}=x)\hbox{ and } \\
&{}&{\mathbb P}_{x}(\overleftarrow{\tau }_{(N)}>n)=
\pi (x)^{-1}\sum_{z\in {\mathbb N}}\pi(z){\mathbb P}_{z}
(\tau _{(N)}>n,X_{n}=x).
\end{eqnarray*}
Then, 
\begin{eqnarray}
\label{eqxy10}
{\mathbb P}_{x}(\overleftarrow{X}_{n}=y\,|\,\overleftarrow{\tau }_{(N)}>n)
&=&\frac{\pi (y){\mathbb P}_{y}(\tau _{(N)}>n,X_{n}=x)}
{\sum_{z\in {\mathbb N}}\pi (z){\mathbb P}_{z}(\tau _{(N)}>n,X_{n}=x)} \\
\nonumber
&=&\frac{\pi (y){\mathbb P}_{y}(X_{n}=x\,|\,\tau _{(N)}>n)}
{\sum_{z\in {\mathbb N}}\pi (z){\mathbb P}_{z}(X_{n}=x\,|\,\tau _{(N)}>n)
\frac{{\mathbb P}_{z}(\tau _{(N)}>n)}{{\mathbb P}_{y}(\tau _{(N)}>n)}}.  
\end{eqnarray}

We recall $P$ monotone means that for $z\ge y$ all $N\ge 1$ it is satisfied 
${\mathbb P}_y(X_1<N)\ge {\mathbb P}_z(X_1<N)$ for all $N\ge 1$. This implies
for all $r\ge 1$ we have ${\mathbb P}_y(X_r<N)\ge {\mathbb P}_z(X_r<N)$ when 
$y\le z$. Monotonicity is also equivalent to the fact that for all
decreasing bounded function $h:{\mathbb N}\to {\mathbb R}$ one has 
${\mathbb E}_y(h(X_1))\ge {\mathbb E}_z(h(X_1))$ when $y\le z$.

\begin{lemma}
\label{lemma1} 
Assume $P$ is monotone. Then, for all $N\ge 1$ one has 
\begin{equation}
\label{eqxy11}
\forall y\le z\le N,\forall n\ge 1: \quad 
{\mathbb P}_{y}(\tau _{(N)}>n)\ge {\mathbb P}_{z}(\tau_{(N)}>n). 
\end{equation}
\end{lemma}

\begin{proof}
Let $N$ be fixed. We will show by recurrence on $k\ge 0$ that for all $n>k$
one has
\begin{equation}
\label{eqxy12}
\PP_y(X_{n-k}<N,..,X_n<N)\ge \PP_z(X_{n-k}<N,..,X_n<N).
\end{equation}
The inequality for $k=0$, $\PP_y(X_n<N)\ge 
\PP_z(X_n<k)$ holds for all $n>0$ because $P$ is monotone 
and $y\le z$.
Let us show it for $k\ge 1$ and all $n>k$. So, we may 
assume we have shown it up to $k-1$ and all $n>k-1$. 
From the Markov property we have
$$
\PP_x(X_{n-k}<N,..,X_n<N)=\sum_{u<N} 
\PP_u(X_1<N,..X_k<N)\PP_x(X_{n-k}=u).
$$
We claim that the function $h$ defined by $h(u)=\PP_u(X_1<N,..,X_k<N)$ 
for $u<N$ and $h(u)=0$ for $u\ge N$, is decreasing in $u$. In fact
this is exactly the induction hypothesis for $k-1$ when one takes 
$n=k$ in (\ref{eqxy12}). Then
$\EE_y(h(X_{n-1})\ge \EE_z(h(X_{n-1})$, which is the
inequality we want to prove: 
$\sum_{u<N} \PP_u(X_1<N)\PP_y(X_{n-1}=u)\ge \sum_{u<N}
\PP_u(X_1<N)\PP_z(X_{n-1}=u)$.

Then, relation (\ref{eqxy11}) is shown because
(\ref{eqxy12}) for $k=n-1$ is equivalent to
$$
\PP_y(\tau_{(N)}\!>\!n)\!=\!\PP_y(X_{n-k}\!<\!N,..,X_n\!<\!N)\ge 
\PP_z(X_{n-k}\!<\!N,..,X_n\!<\!N)\!=\!\PP_z(\tau_{(N)}\!>\!n).
$$
\end{proof}

Let us now show a ratio limit result.

\begin{lemma}
\label{lemma2} 
Let $P$ be a monotone irreducible positive recurrent
stochastic kernel with stationary distribution $\pi ^{\prime }$. We have 
\begin{equation}
\label{eqxy13}
\forall y,z\in {\mathbb N}:\quad \lim\limits_{n,N\to \infty }
\frac{{\mathbb P}_{y}(\tau _{(N)}>n)}
{{\mathbb P}_{z}(\tau _{(N)}>n)}=1. 
\end{equation}
\end{lemma}

\begin{proof}
We will use a recurrence on $n\in \NN$. We will also use the following
remark that follows from Lemma \ref{lemma1}: 
if $\lim\limits_{n, N\to \infty} 
\PP_y(\tau_{(N)}>n)/\PP_z(\tau_{(N)}>n)=1$ for $y<z$, then  
\begin{equation}
\label{eqxy14}
\forall \; y<x<z: \quad  
\lim\limits_{n, N\to \infty}
\frac{\PP_x(\tau_{(N)}>n)}{\PP_y(\tau_{(N)}>n)}=
\lim\limits_{n, N\to \infty}
\frac{\PP_x(\tau_{(N)}>n)}{\PP_z(\tau_{(N)}>n)}=1
\end{equation}
We claim that
\begin{equation}
\label{eqxy15}
\forall x\in \NN: \quad \sum_{z\le x} P(x,z)<1.
\end{equation}
In fact, if $\sum_{z\le x} P(x,z)=1$ holds for some $x$, then  
the monotone property implies $\sum_{z\le x} P(y,z)=1$ for all $y\le x$
and so the set of points $\{1,..,x\}$ is a closed set of $P$ 
contradicting the irreducibility property. Then, the claim holds and 
for all $x\in \NN$ 
there exists some $z>x$ such that $P(x,z)>0$.

Let $x\in \NN$. For $\epsilon>0$ there
exists a bounded $L=L(\epsilon)$ such that 
$\sum_{z\in L}\pi_z>1-(\epsilon/2)$.
From (\ref{eqxy8}) we get that there exists
$n_0, N_0$ such that for $n\ge n_0, N\ge N_0$ one has 
$$
\lim\limits_{n, N\to \infty} \PP_x(X_n\le L \, | \, \tau_{(N)}>n)
>1-\epsilon.
$$
Then,
\begin{equation}
\label{eqxy16}
\forall \, x\in \NN:\quad
\lim\limits_{n, N\to \infty} \PP_x(\tau_{(N)}>n \,| \, 
\tau_{(N)}>n-1)=1.
\end{equation}

For the induction we use the equality
$$
\PP_x(\tau_{(N)}>n)=\sum_{z\in \NN} P(x,z) \PP_z(\tau_{(N)}>n-1),
$$
so
\begin{equation}
\label{eqxy17}
1=\sum_{z\in \NN} P(x,z) 
\frac{\PP_z(\tau_{(N)}>n-1)}{\PP_x(\tau_{(N)}>n-1)} 
\frac{1}{\PP_x(\tau_{(N)}>n \, | \tau_{(N)}>n-1)}.
\end{equation}

We define:
\begin{eqnarray}
\label{eqxy18}
&{}& \hbox{ Property Prop}(x) \hbox{ at } x \hbox{ is}:\\
\nonumber
&{}& \forall y\le x:\;
\lim\limits_{n, N\to \infty} \PP_x(\tau_{(N)}>n)/\PP_y(\tau_{(N)}>n)=1
\hbox{ and } \\
\nonumber
&{}& \exists \, z>x \hbox{ such that } 
\lim\limits_{n, N\to \infty} \PP_x(\tau_{(N)}>n)/\PP_z(\tau_{(N)}>n)=1.
\end{eqnarray}
Note that (\ref{eqxy14}) implies that if Prop$(x)$ holds
then we can always assume $z=x+1$ in (\ref{eqxy18}) 
and so that
$$
\forall \, y\le x: \quad
\lim\limits_{n, N\to \infty} \PP_x(\tau_{(N)}>n)/\PP_{x+1}
(\tau_{(N)}>n)=1.
$$

Let us prove that Prop$(1)$ holds. We need only to show 
that for some $z>x$ one has
$\lim\limits_{n, N\to \infty} \PP_1(\tau_{(N)}>n)/\PP_z(\tau_{(N)}>n)=1$.
From (\ref{eqxy17}) we have
$$
1=\frac{P(1,1)}{\PP_1(\tau_{(N)}\!>\! n \, | \tau_{(N)}\!>\! n\!-\!1)}
+\sum_{z>1} P(1,z) \frac{\PP_z(\tau_{(N)}\!>\! n\!-\!1)}
{\PP_1(\tau_{(N)}\!>\! n\!-\!1)}
\frac{1}{\PP_1(\tau_{(N)}\!>\! n \, | \tau_{(N)}>n\!-\!1)}.
$$
From (\ref{eqxy16}), (\ref{eqxy11})
we deduce that $\lim\limits_{n, N\to \infty} 
\PP_z(\tau_{(N)}>n-1)/\PP_1(\tau_{(N)}>n-1)=1$ for all $z>1$
such that $P(1,z)>0$. From (\ref{eqxy15}) we get Prop$(1)$.

Let $x>1$. We assume Prop$(y)$ holds up to $y=x-1$ and let us 
show Prop$(x)$ is satisfied. We have
\begin{equation}
\label{eqxy19}
1=\sum_{z\in \NN} P(x,z)
\frac{\PP_z(\tau_{(N)}>n-1)}{\PP_x(\tau_{(N)}>n-1)}   
\frac{1}{\PP_x(\tau_{(N)}>n \, | \tau_{(N)}>n-1)}.
\end{equation}
From (\ref{eqxy19}), (\ref{eqxy16}), (\ref{eqxy11})
we get that $\lim\limits_{n, N\to \infty}
\PP_z(\tau_{(N)}>n-1)/\PP_1(\tau_{(N)}>n-1)=1$ for all $z>1$
such that $P(x,z)>0$. From (\ref{eqxy15}) we get Prop$(x)$. 
Then, the result is shown.
\end{proof}

Let us state the Pollak-Siegmund limit relation (\ref{eqxy8}) for the
reversed chain.

\begin{proposition}
\label{propx2} 
We have 
$$
\forall x,y\in {\mathbb N}:\quad \lim\limits_{n,N\to \infty }
{\mathbb P}_{x}(\overleftarrow{X}_{n}=y\,|\,
\overleftarrow{\tau}_{(N)}>n)=\pi (y).
$$
\end{proposition}

\begin{proof}
From condition (\ref{eqxy11}) and (\ref{eqxy8}), (\ref{eqxy13}), 
we can use the dominated convergence theorem
in (\ref{eqxy10}) to get
$$
\lim\limits_{n,N\to \infty} \PP_x(\vX_n=y \,| \, \vtau_{(N)}>n)
=\frac{\pi(y)\pi(x)} {\sum_{z\in \NN} \pi(z)\pi(x)}=\pi(y).
$$
\end{proof}

\subsection{The truncation of the mean expected value have nonnegative dual}
\label{Section12} 

We assume $P$ is monotone on $\NN$. Then, we can consider that a truncation of 
$P$ at level $N$ is a kernel $P_{N}$ taking values in $I_{N}$ that satisfies 
\begin{equation}
\label{eqxy20}
\hbox{ for }x<N:\;P_{N}(x,y)=P(x,y)\hbox{ if }y<N\hbox{ and } P_{N}(x,N)
=\sum_{z\ge N}P(x,z).  
\end{equation}
The unique degree of freedom is to define the redistribution of mass at $N$.
In the truncation $P_N$ we will define, we do it in such a way that the 
stationary distribution is preserved in
a very specific way. Let ${\overline{\pi }}$ be the tail of $\pi $, 
${\overline{\pi }}(x)=\sum_{z\ge x}\pi (z)$, in particular 
${\overline{\pi }}(N)=\sum_{z\ge N}\pi (z)$. We define $P_{N}$ 
by (\ref{eqxy20}) and such that 
$$
P_{N}(N,y)=\frac{1}{{\overline{\pi }}(N)}\sum_{z\ge N}\pi
(z)P(z,y),\;\;P_{N}(N,N)=\frac{1}{{\overline{\pi }}(N)}\sum_{z\ge N}\pi
(z)(\sum_{u\ge N}P(z,u)).
$$

Let $\pi_N$ be given by $\pi_N(x)=\pi(x)$ for $x<N$ and $\pi_N(N)
={\overline{\pi}}(N)$. 
Let us check that $\pi_N$ is the stationary distribution of $P_N$. 
For $y<N$ one has 
$$
\sum_{x=1}^N \pi_N(x) P_N(x,y)= \sum_{x<N} \pi(x) P_N(x,y)+ (\sum_{z\ge N}
\pi(z) P(z,y))=\pi(y)=\pi_N(y),
$$
and for $y=N$, 
\begin{eqnarray*}
\sum_{x=1}^N \pi_N(x) P_N(x,N)&=& \sum_{x<N}\pi(x) (\sum_{u\ge N} P(x,u))+
\sum_{z\ge N} \pi(z) (\sum_{u\ge N} P(z,u)) \\
&=& \sum_{x\in {\mathbb N}}\pi(x) \sum_{u\ge N} P(x,u) 
={\overline{\pi}}(N)=\pi_N(N).
\end{eqnarray*}

This truncation can be written as the action of a mean expected operator.
Let ${\mathbb E}_N$ be the mean expected operator on $L^1({\mathbb N},\pi)$
with respect to the $\sigma-$field induced by the partition 
$\alpha_N=\{\{x\}: x<N\}\cup \{x\ge N\}$. So, 
$$
{\mathbb E}_N f(x)=f(x) \hbox{ if }x<N \hbox{ and } 
{\mathbb E}_N f(x) =
\frac{1}{{\overline{\pi}}(N)}\sum_{y\ge N} \pi(y)f(y) \hbox{ if }x\ge N. 
$$
Since ${\mathbb E}_N f(x)$ is constant for $x>N$ we can identify $\alpha_N$
with $I_N$, the atom $\{x\}$ is identified with $x$ when $x<N$ and the atom 
$\{x\ge N\}$ is identified with $N$.

\begin{proposition}
\label{propx3} 
The truncation $P_{N}$ satisfies 
\begin{equation}
\label{eqxy21}
P_{N}={\mathbb E}_{N}P, 
\end{equation}
where ${\mathbb E}_{N}$ is the mean expected operator defined as above.
\end{proposition}

\begin{proof}
Let $P$ be a stochastic kernel.
Since $\EE_N$ is stochastic then $\EE_N P$ is also stochastic.
It satisfies
\begin{eqnarray*}
&{}&
\EE_N Pf(x)=Pf(x) \hbox{ if }x<N \hbox{ and } \\
&{}&
\EE_N Pf(x)=\frac{1}{\bpi(N)}\sum_{z\ge N} \pi(z)(Pf)(z)
=\frac{1}{\bpi(N)}\sum_{z\ge N} \pi(z) \sum_{y\in \NN} P(z,y)f(y) 
\hbox{ if } x\ge N.
\end{eqnarray*}
For $y<N$ one has
\begin{eqnarray*}
&{}& \EE_N (P{\bf 1}_{\{y\}})(x)=P{\bf 1}_{\{y\}}(x)=P(x,y) 
\hbox{ if } x<N
\hbox{ and }\\
&{}&\EE_N (P{\bf 1}_{\{y\}})(x)=
\frac{1}{\bpi(N)}\sum_{z\ge N} \pi(y) P{\bf 1}_{\{y\}}(z)
=\frac{1}{\bpi(N)}\sum_{z\ge N} \pi(y)P(z,y) \hbox{ if }
x\ge N.
\end{eqnarray*}
Since $P{\bf 1}_{\{y\}}(x)=P(x,y)$ we have proven 
\begin{equation}
\label{eqxy22}
P_N(x,y)=\EE_N P(x,y) \hbox{ for } y<N, x\in \NN.
\end{equation}  
Since $\EE_N P {\bf 1}={\bf 1}=\sum_{y=1}^N {\bf 1}_{\{y\}}$, and
since $P_N$ and $E_N P$ are stochastic operators, from
(\ref{eqxy22}) we conclude $P_N(x,N)=(\EE_N P)(x,N)$. 
Hence, (\ref{eqxy21}) follows.
\end{proof}

Let us now prove that these truncations have a nonnegative dual.

\begin{proposition}
\label{propx4} 
let $P$ be a monotone kernel on $\NN$. Then, $P_{N}$ is a monotone kernel 
having a nonnegative Siegmund 
dual $\widehat{P}_{N}$ with values in $I_{N}$ and such that $N$ is an 
absorbing state.
\end{proposition}

\begin{proof}
We claim that the monotone property on $P$ implies the monotonicity of $P_{N}$. 
Firstly, for all $v,w\in I_N$ we have 
$\sum\limits_{z\le N} P_N(v,z)=1= \sum\limits_{z\le N} P_N(w,z)$.
Now, let $x<N$. For $v\le w<N$ one has 
$$
\sum\limits_{z\le x} P_N(v,z)
= \sum\limits_{z\le x} P(v,z)\le \sum\limits_{z\le x} P(w,z)=\sum\limits_{z\le x} P_N(w,z),
$$ 
and for $v<N$ it holds
\begin{eqnarray*}
\quad\quad \sum\limits_{z\le x} P_N(N,z)&=&
\frac{1}{\bpi(N)}(\sum\limits_{z\le x} \sum_{u\ge N} \pi(u) P(u,z))
=\frac{1}{\bpi(N)} (\sum_{u\ge N}\pi(u)\sum\limits_{z\le x} P(u,z))\\   
&\le &\frac{1}{\bpi(N)} (\sum_{u\ge N}\pi(u)\sum\limits_{z\le x}
P(v,z))=\sum\limits_{z\le x}P(v,z),
\end{eqnarray*}
where the monotonicity of $P$ was used to state the last $\le$ relation.
Then the claim holds, that is $P_{N}$ is monotone.

\medskip

Hence $P_{N}$ has a nonnegative Siegmund
dual $\widehat{P}_{N}$ with values in $I_{N}$ and that following 
(\ref{eqxy4}) it satisfies,
$\hP_N(x,y)=\sum\limits_{z\le x}(P_N(y,z)-P_N(y+1,z))$ for 
$x,y\in I_N$. This gives:
\begin{equation}
\label{eqxy23}
\hP_N(x,y)=\sum\limits_{z\le x}(P(y,z)-P(y+1,z))=\hP(x,y) 
\hbox{ for } x<N, y<N-1,
\end{equation}
and for $x=N$ we get
$$
\hP_N(N,y)=0 \hbox{ for } y\le N-1.
$$
Note that,
\begin{equation}
\label{eqxy24}
\hP_N(x,N-1)=\sum\limits_{z\le x}P(N-1,z)-\sum\limits_{z\le x}P_N(N,z)\ge 0
\hbox{ for } x<N.
\end{equation}
Finally
\begin{equation}
\label{eqxy25}
\hP_N(x,N)=\sum\limits_{z\le x} P_N(N,z)=\frac{1}{\bpi(N)}
\sum\limits_{z\le x} 
(\sum_{u\ge N} \pi(u) P(u,z)) \hbox{ for } x<N;
\end{equation}
and
$$
\hP_N(N,N)=\sum\limits_{z\le N} P_N(N,z)=1.
$$
Hence, $\hP_N$ is a Siegmund dual of $P_N$ with values 
in $I_N$ and $N$ is an absorbing state for $\hP_N$.
\end{proof}

Notice that for $x<N$ one has that the difference between the kernels 
$\widehat P_N(x,y)$ and $\widehat P(x,y)$ only happens at $y=N-1$ and $y=N$.
We have 
\begin{eqnarray*}
&{}&\widehat P_N(x,N-1)+\widehat P_N(x,N)=\sum_{z\le x} P(N-1,z) \hbox{ and }
\\
&{}&\widehat P(x,N-1)+\widehat P(x,N)=\sum_{z\le x} P(N-1,z)+\sum_{z\le x}
P(N+1,z),
\end{eqnarray*}
and so 
$$
(\widehat P(x,N-1)+\widehat P(x,N))-(\widehat P_N(x,N-1)+\widehat P_N(x,N))=
\sum_{z\le x} P(N+1,z)\ge 0.
$$
Therefore, if $\widehat P$ loses mass through $x<N$ then $\widehat P_N$ also
does.

On the other hand it is straightforward to check that the truncation of the
reversed kernel $\overleftarrow{P}_N$ is the reversed of $P_N$ with respect
to $\pi_N$, that is it satisfies 
$$
\overleftarrow{P}_N(x,y)=\pi_N(x)^{-1}P_N(y,x)\pi_N(y), \; x,y\le N.
$$

From (\ref{eqxy5}) we can define the intertwining matrix 
$\Lambda (x,y)=\mathbf{1}(x\ge y)(\pi _{N}(y)/\pi _{N}^{c}(x))$ 
for $x,y\in I_{N}$ where 
$\pi_{N}(y)=\pi (y)$ for $y<N$ and $\pi _{N}(N)={\overline{\pi }}(N)$. The
intertwined matrix $\widetilde{P}_{N}$ of $\overleftarrow{P}_{N}$ which
satisfies $\widetilde{P}_{N}\Lambda _{N}=\Lambda _{N}\overleftarrow{P}_{N}$,
is given by (\ref{eqxy6}). It is 
$\widetilde{P}_{N}(x,y)=\widehat{P}_{N}(x,y)
(\pi_{N}^{c}(y)/\pi _{N}^{c}(x))$ for $x,y\in I_{N}$. Note that 
$\pi_{N}^{c}(y)=\pi ^{c}(y)$ for $y<N$ and $\pi _{N}^{c}(N)=1$.

\begin{proposition}
\label{propx5} 
Let us consider two truncations as above, $P_{N}$ and $P_{N+1}
$ at levels $N$ and $N+1$, respectively. Then, the time 
$\widetilde{\tau}_{N}^{N}$ of hitting $N$ by $\widetilde{P}_{N}$ 
is stochastically smaller
than the time $\widetilde{\tau }_{N+1}^{N+1}$ of hitting $N+1$ by 
$\widetilde{P}_{N+1}$.
\end{proposition}

\begin{proof}
We will note by  $\hX^{N}=(\hX^{N}_{n})$ and $\hX^{N+1}=(\hX^{N+1}_{n})$ 
the Markov chains associated to kernels $\hP_N$ and $\hP_{N+1}$, 
respectively.  

From (\ref{eqxy24}), for every $x<N-1$ one has
$$
\hP_N(x,N-1)-\hP_{N+1}(x,N-1)=\sum\limits_{z\le x}P(N,z)
-\frac{1}{\bpi(N)}(\sum_{u\ge N} \pi(u) \sum\limits_{z\le x}P(u,z)).
$$
Similarly,
$$
\hP_{N+1}(x,N)=\sum\limits_{z\le x}P(N,z)-
\frac{1}{\bpi(N+1)} (\sum_{u\ge N+1} \pi(u) \sum\limits_{z\le x} P(u,z)).
$$
Then, by monotonicity
$$
\hP_N(x,N-1)-\hP_{N+1}(x,N-1)\le \hP_{N+1}(x,N).
$$
So,
\begin{eqnarray*}
&{}&
(\hP_{N+1}(x,N-1)+\hP_{N+1}(x,N-1))-\hP_N(x,N)\\
&{}&=\frac{1}{\bpi(N)}(\sum_{u\ge N} \pi(u) \sum\limits_{z\le x}P(u,z))-
\frac{1}{\bpi(N+1)} 
(\sum_{u\ge N+1} \pi(u) \sum\limits_{z\le x} P(u,z)).
\end{eqnarray*}
from (\ref{eqxy25}) we also have
\begin{eqnarray*}
&{}& 0\le \hP_N(x,N)-\hP_{N+1}(x,N+1)\\
&{}&=\frac{1}{\bpi(N)}
(\sum_{u\ge N} \pi(u) \sum\limits_{z\le x}P(u,z))-
\frac{1}{\bpi(N+1)}
(\sum_{u\ge N+1} \pi(u) \sum\limits_{z\le x} P(u,z)),
\end{eqnarray*}
where the nonnegativity follows from monotonicity of $P$.
Then,
$$
\hP_{N+1}(x,N-1)+\hP_{N+1}(x,N)+\hP_{N+1}(x,N+1)
=\hP_N(x,N-1)+\hP_N(x,N).
$$
From the above equalities and inequalities and by using (\ref{eqxy23}) at every
step when we are in some state $y<N-1$, we can make a coupling
between both chains $\hX_N$ and $\hX_{N+1}$ such that when
both chains start from $x<N-1$ we have
\begin{eqnarray*}
&{}& \hX_{N+1}\in \{N,N+1\} \Leftrightarrow \hX_N=N \hbox{ and }
\\
&{}& \forall y\le N-1,\; \hX_{N+1}=y \Leftrightarrow \hX_N=y.
\end{eqnarray*}

On the other hand from (\ref{eqxy23}) we get
$$
\hP_N(N-1,y)=\sum\limits_{z\le x}(P(y,z)-P(y+1,z))=\hP(x,y)=\hP_{N+1}(N-1,y),
$$ 
so the distribution to $y\le N-1$ is the same for the two kernels.
Moreover
$$
\hP_N(N-1,N)=\sum\limits_{z\le N-1} P(N,z),
$$
and
\begin{eqnarray*}
&{}&
\hP_{N+1}(N-1,N)+\hP_{N+1}(N-1,N+1)\\
&{}&=
\sum\limits_{z\le N-1} P(N,z)-\sum\limits_{z\le N-1}P_{N+1}(N+1,z)+
\sum\limits_{z\le N-1} P_{N+1}(N+1,z).
\end{eqnarray*}
Therefore
$$
\hP_N(N-1,N)=\hP_{N+1}(N-1,N)+\hP_{N+1}(N-1,N+1).
$$
Hence, once both chains start from $N-1$, they can be coupled
to return to some state $y\le N-1$, or, if not, the rest of the mass
of kernel $\hP_N$ moves to the absorbing state $N$, and for $\hP_{N+1}$
part of this mass moves to the absorbing state $N+1$ while the rest goes
to $N$.

We have shown that the absorption time $\htau^{N}_N$
of $\hP_N$ at level $N$, is smaller
than the absorption time $\htau^{N+1}_{N+1}$ of $\hP_{N+1}$ at 
level $N+1$, that is
$$
\forall x\le N-1, k>0:\;\,
\PP_x(\htau^{N}_N\le k)\ge \PP_x(\htau^{N+1}_{N+1}\le k).
$$
Now for $x\le N-1$ one has, $\pi_N^{c}(x)=\pi^{c}(x)=\pi_{N+1}^{c}(x)$
and $\pi_N^{c}(N)=1=\pi_{N+1}^{c}(N+1)$, so
\begin{eqnarray*}
&{}&
\PP_x(\wX^{N}_1=x_1,..,\wX^{N}_k=N)
=\frac{1}{\pi^{c}(x)}\PP_x(\hX^{N}_1=x_1,..,\hX^{N}_k=N),\\
&{}& \PP_x(\wX^{N+1}_1=x_1,..,\wX^{N+1}_k=N)
=\frac{1}{\pi^{c}(x)}\PP_x(\hX^{N+1}_1=x_1,...,\hX^{N+1}_k=N).
\end{eqnarray*}
Then,
$$
\PP_x(\wtau^{N}_N=j)=\frac{1}{\pi^{c}(x)}\PP_x(\htau^{N}_N=j)
\hbox{ and }  
\PP_x(\wtau^{N+1}_{N+1}=j)=
\frac{1}{\pi^{c}(x)}\PP_x(\htau^{N+1}_{N+1}=j).
$$
Hence, we conclude
$$
\forall x\le N-1, k>0:\;\,
\PP_x(\wtau^{N}_N\le k)\ge \PP_x(\wtau^{N+1}_{N+1}\le k),
$$
and so $\wtau^{N}_N$ is stochastically smaller than $\wtau^{N+1}_{N+1}$.
\end{proof}

\section{Strong Stationary Times}
\label{SectionSST} 

Let $\pi_{0}^{\prime }$ be the initial distribution of 
$X_{0}$, so $\pi _{n}^{\prime }=\pi _{0}^{\prime }P^{n}$ is the distribution
of $X_{n}$. A stopping time is noted $T_{\rho }$ when $X_{T_{\rho }}\sim
\rho $. We recall $\tau _{a}=\inf \{n\ge 0:X_{n}=a\}$ for $a\in I$. When one
wants to emphasize the initial distribution $\pi _{0}^{\prime }$ of $X_{0}$,
these times are written by $T_{\rho }^{\pi _{0}}$ and 
$\tau _{a}^{\pi _{0}}$, respectively.

A stopping time $T$ is called a strong stationary time if $X_T\sim
\pi^{\prime}$ and it is independent of $T$, see \cite{AD}. The separation
discrepancy is defined by 
$$
\text{sep}\left({\pi}_n,{\pi }\right) := \sup\limits_{y\in I} 
\left[ 1-\frac{\pi_n(y)}{\pi(y)}\right] . 
$$
In Proposition $2.10$ in \cite{AD} it was proven that every strong
stationary time $T$ satisfies 
\begin{equation}  
\label{eqxy26}
\forall \, n\ge 0\,:\;\;\; \text{sep}\left(\pi_n,\pi\right) \leq 
{\mathbb P}_{\pi_0}\left(T>n\right) \,.
\end{equation}
In Proposition $3.2$ in \cite{AD} it was shown that there exists a strong
stationary time $T$, called sharp, that satisfies equality in 
(\ref{eqxy26}), 
$$
\forall \, n\ge 0\,:\;\;\; \text{sep}\left({\pi}_{n},{\pi }\right) = 
{\mathbb P}_{\pi_{0}}\left(T>n\right) \,. 
$$

Assume we are in the framework of Theorem \ref{theo1}, so 
$\widetilde{P}\Lambda =\Lambda \overleftarrow{P}$. 
A random time for $\widetilde{X}$ is
noted by $\widetilde{T}$ and we use similar notations as those introduced
for random times $T$ for $X$. The initial distributions of 
${\overleftarrow{X}}_{0}$ and $\widetilde{X}_{0}$ are respectively noted 
by $\overleftarrow{\pi }_{0}^{\prime}$ and 
${\widetilde{\pi }_{0}^{\prime}}$. We assume they
are linked, this means: 
\begin{equation}
\label{eqxy27}
\overleftarrow{\pi }_{0}^{\prime }={\widetilde{\pi }}_{0}^{\prime }\Lambda .
\end{equation}
In this case the intertwining relation $\widetilde{P}^{n}\Lambda 
=\Lambda {\overleftarrow{P}}^{n}$ implies 
$\overleftarrow{\pi }_{n}^{\prime }={\widetilde{\pi }}_{n}^{\prime }\Lambda$ 
for $n\ge 0$, where $\overleftarrow{\pi}_{n}$ and 
${\widetilde{\pi }}_{n}$ are the distributions of 
$\overleftarrow{X}_{n}$ and $\widetilde{X}_{n}$ respectively.

Since $\Lambda $ is stochastic it has a left probability eigenvector 
$\pi_{\Lambda }^{\prime }$, so $\pi _{\Lambda }^{\prime }=
\pi_{\Lambda}^{\prime }\Lambda $ and $\pi_{\Lambda }$ is linked with 
itself. If $\Lambda $ is non irreducible then ${\pi}_{\Lambda }$ could 
fail to be strictly positive, which is the case for the Siegmund kernel 
where $\Lambda$ is given by (\ref{eqxy5}) and one can check that 
$\mathbf{e}_{1}$ is the
unique left eigenvector satisfying $\mathbf{e}_{1}^{\prime }
=\mathbf{e}_{1}^{\prime}\Lambda $. So, the initial conditions 
$\widetilde{X}_{1}\sim \delta_{1}$ and $\overleftarrow{X}_{0}\sim \delta_{1}$ 
are linked. Assume $P$ is monotone. From relation (\ref{eqxy5}) one gets 
that (\ref{eqxy27}) is
equivalent to $\overleftarrow{\pi }_{0}(x)/\pi (x)=
\sum_{y\ge x}{\widetilde{\pi}}(y)/\pi ^{c}(y)$ for all $x\in I$. 
(See relation (4.7) and (4.10) in 
\cite{DF}). In the finite case $I=I_{N}$, Corollary \ref{coro1} $(iii)$
states that if $P$ is monotone then ${\widetilde{\partial }}=N$ is the
unique absorbing state for ${\widetilde{X}}$. Let us now introduce the sharp
dual.

\begin{definition}
\label{def3} 
The process $\widetilde{X}$ is a sharp dual to 
${\overleftarrow{X}}$ if it has an absorbing state ${\widetilde{\partial }}$, 
and when ${\overleftarrow{X}}$ and $\widetilde{X}$ start from 
linked initial conditions 
$\widetilde{X}_{0}\sim \widetilde{\pi }_{0}$, 
${\overleftarrow{X}}_{0}\sim \overleftarrow{\pi}_{0}$ 
with $\overleftarrow{\pi }_{0}^{\prime }=\widetilde{\pi }_{0}^{\prime }\Lambda$, 
then it holds 
$$
\text{sep}(\overleftarrow{\pi }_{n},\pi )
={\mathbb P}_{{\widetilde{\pi }}_{0}}(\widetilde{\tau }_{\widetilde{\partial }}>n),
\;\forall \,n\ge 0.
$$
$\Box $
\end{definition}

We recall that if $\widetilde{\partial}$ is an absorbing state for 
$\widetilde{P}$, then $\pi^{\prime}=
\mathbf{e}^{\prime}_{\widetilde{\partial}} 
\Lambda$ (this is (\ref{eqxy2})).

We now state the sharpness result alluded to in Remark $2.39$ of \cite{DF}
and in Theorem 2.1 in \cite{F}. The hypotheses stated in Remark $2.39$ are
understood as the condition (\ref{eqxy28}) below. The results of this
section were proven in \cite{HM1}.

We recall the definition made in Section $4$ in \cite{mb}: A state $d\in I$
is called separable for $X$ when 
$$
\text{sep}(\pi_{n},\pi )=1-\frac{\pi_{n}(d)}{\pi(d)}, \, n\ge 1.
$$
On the other hand, $d\in I$ was called a witness state if it satisfies 
\begin{equation}  
\label{eqxy28}
\Lambda \mathbf{e}_{d}=\pi (d)\mathbf{e}_{\widetilde{\partial }}.
\end{equation}
We note that in the monotone finite case in $I_N$, the state $d=N$ is a
witness state, this is exactly (\ref{eqxy7}) in Corollary \ref{coro1}.

\begin{proposition}
\label{propx6} 
Let $X$ be an irreducible positive recurrent Markov chain, 
$\widetilde{X}$ be a $\Lambda -$intertwining of ${\overleftarrow{X}}$ having 
$\widetilde{\partial }$ as an absorbing state. If $d\in I$ is a witness state
then $d$ is a separable state, $\widetilde{X}$ is a sharp dual to 
${\overleftarrow{X}}$ and it is satisfied, 
$$
\text{sep}(\overleftarrow{\pi }_{n},\pi )=
1-\frac{\overleftarrow{\pi }_{n}(d)}{\pi(d)}
={\mathbb P}_{{\widetilde{\pi }}_{0}}
(\widetilde{\tau }_{\widetilde{\partial }}>n),\;n\ge 1.
$$
$\Box $
\end{proposition}

\begin{remark}
\label{remk4} 
Since we have shown that being witness implies being
separable, Corollary $4.1$ of \cite{mb} stated for separable states applies
for a witness state, this is ${\overleftarrow{\tau}}_{d}^{\pi_0}=
{\overleftarrow{T}}_{\pi}^{\pi_{0}}+Z$ is an independent 
sum and $Z\sim {\overleftarrow{\tau}}_{d}^{\pi}$. $\Box $
\end{remark}

\begin{remark}
\label{rmk5} 
The condition of being witness can be stated in terms of the
dual function $H$. In fact, in \cite{HM} it was shown that if there exists 
$\widehat{a}\in \widehat{I}$ and $d\in I$ such that for some constants 
$c_{1}>0$, $c>0$ one has 
$$
H\mathbf{e}_{\widehat{a}}=c_{1}\mathbf{1}\,\text{ and }
\mathbf{e}_{d}^{\prime}H=c\mathbf{e}_{\widehat{a}}^{\prime }.
$$
Then $d$ is a witness state and $\widetilde{X}$ is a sharp dual to 
$\overleftarrow{X}$. $\Box$
\end{remark}

\begin{corollary}
\label{coro2} 
For a monotone irreducible stochastic kernel $P$, the 
$\Lambda-$intertwining Markov chain $\widetilde{X}$ 
has $N$ as an absorbing state
and it is a sharp dual of $\overleftarrow{X}$. Also, $N$ is a separable
state and the initial conditions $\overleftarrow{X}_{0}=\delta _{1}$ and 
$\widetilde{X}=\delta_{1}$ are linked. $\Box $
\end{corollary}

\section{The Diaconis-Fill Coupling}
\label{SectionDFC}

Let $\widetilde{P}$ be a $\Lambda-$intertwining of $P$. Consider the
following stochastic kernel $\underline{P}$ defined on $I\times \widetilde{I}
$, which was introduced in \cite{DF}, 
\begin{equation}  
\label{eqxy29}
\underline{P}\left((x,\widetilde{x}),(y,\widetilde{y})\right)= 
\frac{P(x,y)\,\widetilde{P}(\widetilde{x},\widetilde{y})\, 
\Lambda (\widetilde{y},y)} {(\Lambda P)(\widetilde{x},y)} 
\mathbf{1} \left( (\Lambda P)(\widetilde{x},y)>0\right).
\end{equation}

Let $\underline{X}=(\underline{X}_{n}:n\ge 0)$ be the chain taking values in 
$I\times \widetilde{I}$, evolving with the kernel $\underline{P}$ and with
initial distribution 
$$
{\mathbb P}\left(X_0=x_0, \widetilde{X}_0=\widetilde{x}_0\right)= 
{\underline{\pi}}_0(x_0,{\widetilde x}_0)=\pi^{\prime}_0({\widetilde x}_0)
\Lambda(\widetilde{x_0},x_0),\; x_0\in I,{\widetilde x}_0\in \widetilde I, 
$$
where ${\widetilde{\pi}}^{\prime}_0$ is an initial distribution of 
$\widetilde{X}$. In \cite{DF} it was proven that $\underline{X}$ starting
from ${\underline{\pi}}^{\prime}_0$ is a coupling of the chains $X$ and 
$\widetilde{X}$ starting from 
$\pi^{\prime}_0={\widetilde{\pi}}^{\prime}_0 \Lambda$ and 
${\widetilde{\pi}}^{\prime}_0$, respectively. Since 
$X$ and $\widetilde{X}$ are the components of $\underline{X}$ 
one puts $\underline{X}_{n}=(X_{n},\widetilde{X}_{n})$.

Let ${\mathbb P}$ be the probability measure on 
$(I\times {\widetilde I})^\NN $ induced by the coupling transition kernel 
$\underline{P}$ and assume 
$\underline{X}$ starts from linked initial conditions. In \cite{DF} (also 
\cite{cpy}) it was shown, 
\begin{equation}  
\label{eqxy30}
\forall \, n\ge 0\,:\;\;\; \Lambda (\widetilde{x}_{n},x_{n})= 
{\mathbb P}\left(X_{n}=x_{n}\mid \widetilde{X}_{n}=
\widetilde{x}_{n}\right) ={\mathbb P}
\left(X_{n}=x_{n}\mid \widetilde{X}_{0}= \widetilde{x}_{0} \cdots 
\widetilde{X}_{n}= \widetilde{x}_{n}\right)\,.
\end{equation}
Hence, the equality $\pi^{\prime}=\mathbf{e}^{\prime}_{\widetilde{\partial}}
\Lambda$ in (\ref{eqxy2}) together with relation (\ref{eqxy30}) give 
$$
\pi(x)=\Lambda (\widetilde{\partial },x)= {\mathbb P}(X_{n}=
x\mid \widetilde{X}_{0}= \widetilde{x}_{0}\cdots \widetilde{X}_{n}= 
\widetilde{\partial }). 
$$
Then, the following result shown in \cite{DF} holds.

\begin{theorem}
\label{theo2} 
Let $X$ be an irreducible positive recurrent Markov chain with
stationary distribution $\pi ^{\prime }$ and let $\widetilde{X}$ be a 
$\Lambda -$intertwining of $X$. Assume the initial conditions are linked,
meaning $\pi _{0}^{\prime }={\widetilde{\pi }}_{0}^{\prime }\Lambda $. 
Then, 
${\widetilde{X}}$ is called a strong stationary dual of $X$, which means
that the following equality is satisfied, 
$$
\pi(x)={\mathbb P}\left( X_{n}=x\mid \widetilde{X}_{0}=
\widetilde{x}_{0}\cdots \widetilde{X}_{n-1}=\widetilde{x}_{n-1},
\widetilde{X}_{n}=\widetilde{\partial }\right) ,\,\forall x\in I,\,n\ge 0,
$$
where $\widetilde{x}_{0}\cdots \widetilde{x}_{n-1}\in \widetilde{I}$ satisfy 
${\mathbb P}\left( \widetilde{X}_{0}=\widetilde{x}_{0}\cdots 
\widetilde{X}_{n-1}=\widetilde{x}_{n-1},\widetilde{X}_{n}=
\widetilde{\partial }\right) >0$. $\Box $
\end{theorem}

\subsection{Quasi-stationarity and coupling}

Below we state a property on quasi-stationarity of the coupling. Let us
recall some elements on quasi-stationarity (see for instance \cite{cmsm}).
Let $Y=(Y_{n}:n\ge 0)$ be a Markov chain with values on a countable set $J$
and transition stochastic kernel $Q$. Let $K$ be a nonempty strictly subset
of $J$ and let $\tau_{K}$ be the hitting time of $K$. A probability measure 
$\mu $ on $J\setminus K$ is a quasi-stationary distribution (q.s.d.) for $Y$
and the forbidden set $K$ if 
\begin{equation}
\label{eqxy31}
{\mathbb P}_{\mu }(Y_{n}=j\,|\,\tau _{K}>n)=\mu (j),\;j\in J\setminus K.
\end{equation}
Note that the q.s.d. does not depend on the behavior of the chain $Y$ on 
$K$, so we assume $Y$ is absorbed at $K$. In order that $\mu $ is a q.s.d. 
it suffices to satisfy (\ref{eqxy31}) for $n=1$, which is equivalent for 
$\mu$ being a left eigenvector of $Q$, so 
$$
\mu ^{\prime }Q=\gamma \mu ^{\prime }.
$$
In the finite case, $\mu $ is the normalized left Perron-Frobenius
eigenvector with Perron-Frobenius eigenvalue $\gamma $. It is easily checked
that 
$$
\gamma ={\mathbb P}_{\mu }(\tau _{K}>1)=\sum_{j\in J\setminus K}
\mu_{j}\sum_{j^{\prime }\in J\setminus K}P(j,j^{\prime }). 
$$
The hitting time 
$\tau _{K}$ starting from $\mu $ is geometrically distributed: 
${\mathbb P}_{\mu }(\tau _{K}>n)=\gamma ^{n}$, this is why $\gamma$ 
is called the survival decay rate. If the 
chain is irreducible in $J\setminus K$, then
every q.s.d. is strictly positive and for all $j\in J\setminus K$ one has 
${\mathbb P}_{j}(\tau _{K}>n)\le C_{j}\gamma ^{n}$ 
with $C_{j}=\mu _{j}^{-1}$.

In the next result we put in relation the q.s.d. of the process 
${\widetilde{X}}$ with forbidden state $\widetilde{\partial }$ and 
the q.s.d. of the
process $(X,{\widetilde{X}})$ with forbidden set 
${\underline{\partial }}=I\times \{\widetilde{\partial }\}$. 
Since ${\widetilde{X}}=\widetilde{\partial }$ 
is equivalent to $(X,{\widetilde{X}})
\in {\underline{\partial }}$, then it is straightforward 
that the survival decay rates for ${\widetilde{X}}$ 
and $(X,{\widetilde{X}})$ with respect to $\widetilde{\partial }$ and 
${\underline{\partial }}$ respectively, are the same (that is the
Perron-Frobenius eigenvalue is common for both processes).

\begin{proposition}
\label{propx7} 
Assume ${\widetilde{\mu }}^{\prime }$ is a q.s.d. for the
process ${\widetilde{X}}$ with the forbidden state $\widetilde{\partial}$.
Then, the probability measure 
$$
{\underline{\mu }}(x_{0},{\widetilde{x}}_{0})={\widetilde{\mu }}
({\widetilde{x}}_{0})\Lambda 
({\widetilde{x}}_{0},x_{0}),\;(x_{0},{\widetilde{x}}_{0})\in
I\times (\widetilde{I}\setminus \{{\widetilde{\partial }}\}),
$$
is a q.s.d. for $(X,{\widetilde{X}})$ with forbidden set 
${\underline{\partial}}=I\times \{\widetilde{\partial }\}$.
\end{proposition}

\begin{proof}
From the hypothesis we have
$$
\sum_{\wx\in \wI, \wx\neq \wpartial} \wmu(\wx) \wP(\wx,\wy)
=\gamma \wmu(\wy),\; \wy\in  \wI\setminus \{\wpartial\},
\hbox{ with }
\gamma=1-\sum_{\wx\in \wI, \wx\neq \wpartial} 
\mu(\wx)\wP(\wx,\wpartial).
$$
Now
\begin{eqnarray*}
&{}& 
\sum_{x\in I}\sum_{\wx\in \wI\setminus \{\wpartial\}}\wmu(\wx)
\Lambda(\wx,x) \uP((x,\wx),(y,\wy))\\
&{}&=\sum_{x\in I} \sum_{\wx\in  \wI\setminus \{\wpartial\}}
\wmu(\wx)\Lambda(\wx,x)P(x,y)\wP(\wx,\wy)\Lambda(\wy,y)
{\bf 1}((\Lambda P)(\wx,y)>0)\Lambda(\wx,y)^{-1}\\
&{}& =
\sum_{\wx\in  \wI\setminus \{\wpartial\}}\wmu(\wx)\wP(\wx,\wy)
\Lambda(\wy,y)
{\bf 1}((\Lambda P)(\wx,y)>0)(\Lambda P)(\wx,y)^{-1}
(\sum_{x\in I} \Lambda(\wx,x) P(x,y))\\
&{}& =\sum_{\wx\in  \wI\setminus \{\wpartial\}}
\wmu(\wx)\wP(\wx,\wy) \Lambda(\wy,y) {\bf 1}((\Lambda P)(\wx,y)>0).
\end{eqnarray*}
Now, if $P(\wx,\wy)>0$ and $\Lambda(\wy,y)>0$, then we have
$(\Lambda P)(\wx,y)=(\wP \Lambda)(\wx,y)>0$ and so we get
\begin{eqnarray*}
\sum_{x\in I} \sum_{\wx\in  \wI\setminus \{\wpartial\}} 
\wmu'(\wx)\underline{P}((x,\wx),(y,\wy))&=&
\sum_{\wx\in  \wI\setminus \{\wpartial\}}
\wmu(\wx)\wP(\wx,\wy)\Lambda(\wy,y)\\
&{}&=\gamma \wmu(\wy)\Lambda(\wy,y)=\gamma \umu(y,\wy). 
\end{eqnarray*} 
Hence, $\umu$ is a q.s.d. for $\uX$ with forbidden set 
$\upartial$. 
\end{proof}

\begin{remark}
\label{remk6} 
Based upon basic relations on quasi-stationarity (see Theorem $2.6$ 
in \cite{cmsm}), it can be shown that starting from ${\underline{\mu }}$
the random variables 
${\underline{X}}_{\tau_{{\underline{\partial }}}}
=(X_{\tau _{{\underline{\partial }}}},{\underline{\partial }})$ and 
$\tau_{{\underline{\partial }}}$ are independent, 
so $X_{\tau _{{\underline{\partial }}}}$ 
and $\tau_{{\underline{\partial }}}$ are independent. But in
the setting of the Diaconis-Fill coupling this property is contained in
Theorem \ref{theo2}. In fact, the latter result ensures a much stronger
result which is that starting from any linked initial condition 
${\underline{\pi}}_{0}$ one has that 
$X_{\tau _{{\underline{\partial }}}}$ and 
$\tau_{{\underline{\partial}}}$ are independent and 
$X_{\tau _{{\underline{\partial}}}}\sim \pi ^{\prime }$. $\,\Box $
\end{remark}

Let $P$ be monotone on $I_N$. From Corollary \ref{coro1} $(i)$, 
$\Lambda({\widetilde x},y)=\mathbf{1}({\widetilde x}\ge y) 
\pi(y)/\pi^{c}({\widetilde x})$ and 
$\widetilde{P}({\widetilde x},{\widetilde y})=\widehat{P} 
({\widetilde x},{\widetilde y})\pi^{c}({\widetilde y})/
\pi^{c} ({\widetilde x})$. 
The coupling (\ref{eqxy29}) for $\overleftarrow{P}$ satisfies, 
$$
\Lambda \overleftarrow{P}({\widetilde x},y)=\sum_{z\le {\widetilde x}} 
\frac{\pi(z)}{\pi^{c}({\widetilde x})} P(y,z)
\frac{\pi(y)}{\pi(z)}= \frac{\pi(y)}{\pi^{c}({\widetilde x})}
(\sum_{z\le {\widetilde x}} P(y,z)). 
$$
Then, $\Lambda \overleftarrow{P}({\widetilde x},y)>0$ is equivalent to 
$\sum_{z\le {\widetilde x}} P(y,z)>0$, for ${\widetilde x}, y\in I$, so 
\begin{eqnarray*}
&{}& \underline{\overleftarrow{P}} ((x,\widetilde{x}),(y,\widetilde{y})) 
=\frac{\overleftarrow{P}(x,y)\, \widetilde{P}(\widetilde{x},\widetilde{y})\,
\Lambda(\widetilde{y},y)} {(\Lambda \overleftarrow{P})(\widetilde{x},y)} 
\mathbf{1} \left( (\Lambda P)(\widetilde{x},y)>0\right) \\
&{}& =\frac{\pi(y)}{\pi(x)} \mathbf{1}({\widetilde y} \ge y) P(y,x) 
\widehat{P}({\widetilde x},{\widetilde y}) 
\mathbf{1}(\sum_{z\le {\widetilde x}}P(y,z)>0) 
(\sum_{z\le {\widetilde x}} P(y,z))^{-1}.
\end{eqnarray*}

Now, in this coupling we can set the truncations kernel $P_N$ of the mean
expected value, whose reversed time kernel satisfies 
$\overleftarrow{P}_N(x,y)=\pi_N(x)^{-1}P_N(y,x)\pi_N(y)$. 
From Proposition \ref{propx5} one
gets that the time for $\widetilde P_N$ to attain the absorbing state $N$ is
stochastically smaller that the time for $\widetilde P_{N+1}$ to attain 
$N+1$. Then, the time for $\overleftarrow{P}_N$ to attain the stationary
distribution $\pi_N$ is stochastically smaller than the time for 
$\overleftarrow{P}_ {N+1}$ to attain $\pi_{N+1}$.

\section*{Acknowledgements}

The authors acknowledge the partial support given by the CONICYT BASAL-CMM
project AFB170001. S. Mart\'{i}nez thanks the hospitality of {Laboratoire de Physique
Th\'{e}orique et Mod\'{e}lisation at the Universit\'{e} de Cergy-Pontoise. }

\end{document}